\documentclass[11pt,reqno]{amsart}
\usepackage[dvips]{graphicx,color}
\usepackage{amssymb,amscd,latexsym,epsfig}
\usepackage[mathscr]{euscript}

\DeclareFontFamily{OT1}{rsfs}{}
\DeclareFontShape{OT1}{rsfs}{n}{it}{<-> rsfs10}{}
\DeclareMathAlphabet{\curly}{OT1}{rsfs}{n}{it}

\newcommand\I{\mathscr I}
\newcommand\E{\mathscr E}
\newcommand\X{\mathscr X}
\newcommand\Y{\mathscr Y}
\newcommand\PP{\mathbb P}
\newcommand\C{\mathbb C}
\newcommand\Q{\mathbb Q}
\newcommand\R{\mathbb R}
\newcommand\Z{\mathbb Z}
\newcommand\N{\mathbb N}

\newcommand{\Rt}[1]{\stackrel{#1\,}{\longrightarrow}}

\newcommand\Into{\ar@{^{ (}->}[r]}

\newfont{\bigtimesfont}{cmsy10 scaled \magstep5}
\newcommand{\bigtimes}{\mathop{\lower0.9ex\hbox{\bigtimesfont\symbol2}}}

\newcommand\rk{\operatorname{rank}}
\newcommand\Tr{\operatorname{tr}}

\newcommand\im{\operatorname{im}}

\newcommand\Ext{\operatorname{Ext}}

\newcommand\beq[1]{\begin{equation}\label{#1}}
\newcommand\eeq{\end{equation}}
\newcommand\beqa{\begin{eqnarray*}}
\newcommand\eeqa{\end{eqnarray*}}

\makeatletter \@addtoreset{equation}{section} \makeatother

\newtheorem{defn}[equation]{Definition}
\newtheorem{thm}[equation]{Theorem}
\newtheorem{lem}[equation]{Lemma}

\newtheorem{prop}[equation]{Proposition}

\title[]{A simple limit for slope instability}
\author{J. Stoppa and E. Tenni}
\begin{document}
\begin{abstract}\noindent
Ross and Thomas have shown that subschemes can K-destabilise polarised varieties, yielding a notion known as slope (in)stability for varieties. Here we describe a special situation in which slope instability for \emph{varieties} (for example of general type) corresponds to a slope instability type condition for certain \emph{bundles}, making the computations almost trivial.
\end{abstract}
\maketitle
\section{Introduction}
Let $X$ be a normal projective algebraic variety over $\C$. For any ample line bundle $L$ on $X$ and any closed subscheme $Z \subset X$ we can form the Hilbert-Samuel polynomial for $k \gg 1$
\[\chi(L^{k}\otimes \I^{k x}_Z) = \alpha_0(x)k^n + \alpha_1(x)k^{n-1} + o(k^{n-1})\]
where $\alpha_i(x), i = 0,1$ are understood as polynomials in $x \in \R$ (while the right hand side only makes sense for $k x \in \N$).

Recall that the Seshadri constant $\epsilon(Z, L)$ is defined as the supremum of rational $c$ for which $L^{k}\otimes \I^{k c}_Z$ is globally generated for all large $k$ such that $k c \in \N$ .

For $c \in (0, \epsilon(Z,L)]$ Ross and Thomas \cite{rt_diff} consider the integral
\[\mu_c(Z,L) = \frac{\int^c_0 \left(\alpha_1(x) + \frac{\alpha'_0(x)}{2}\right)dx}{\int^c_0 \alpha_0(x)dx}\]
and study the inequality
\beq{destab}
\mu_c(Z,L) > \mu(X,L) := \frac{\alpha_1(0)}{\alpha_0(0)}
\eeq
reminiscent of slope instability for sheaves. If the latter inequality holds for some $c \in (0, \epsilon(Z,L)]$ we say that $Z$ slope-destabilises the polarised variety $(X, L)$.

The main consequence of slope instability is K-instability of the polarisation $L$ in the sense of Tian and Donaldson. This is most meaningful when $X$ is smooth, since then the deep results of \cite{chen_tian}, \cite{don_calabi} imply a non-existence result for K\"ahler metrics of constant scalar curvature representing the cohomology class $c_1(L)$. (For a differential-geometric interpretation of slope instability see \cite{twisted}). From an algebro-geometric point of view the main consequence is asymptotic instability, see \cite{rt_alg} Theorem 3.9.

Much progress has been made in understanding slope instability in special situations. For example it has been proved recently by Ross and Panov that exceptional divisors of high genus always slope-destabilises suitable polarisations on smooth algebraic surfaces \cite{dima} (a special case of this had been used previously by Ross to give the first example of an asymptotically unstable general type surface~\cite{ross}).\\

However in general the integral $\mu_c(Z,L)$ is hard to compute, making it difficult to find destabilising subschemes in a systematic way. We are interested in special situations in which one can prove inequality \eqref{destab} without having to evaluate the integral explicitly. In this paper we describe one such situation. More precisely we find a class of varieties for which (Ross-Thomas, nonlinear) slope instability corresponds to (a suitable, linear modification of) slope instability of certain bundles, making the computations almost trivial.\\

Let $E$ be a locally free sheaf over a smooth projective curve $C$ of genus $g$. We assume that the relative Serre line bundle $\mathcal{O}_{\PP(E)}(1)$ on the projective bundle $\pi\!:\PP(E)\to C$ of \emph{lines} is globally generated. For $m \geq 2$ let $X$ be a normal element of the linear system $|\mathcal{O}_{\PP(E)}(m)|$ (more generally let $X$ be a normal complete intersection of elements of $|\mathcal{O}_{\PP(E)}(m_i)|$, $i = 1, \dots, r$). We will see that for any ample $\Q$-divisor $A$ on $C$ the $\Q$-line bundle
\[\mathcal{L}_A = (\mathcal{O}_{\PP(E)}(1)\otimes\pi^* A)|_X\]
is ample. Let $F \subset E$ be a locally free subsheaf. We wish to study the slope (in)stability of the polarised variety $(X, \mathcal{L}_A)$ with respect to the closed subscheme $\PP(F)\cap X \subset X$ (the scheme-theoretic intersection; we assume this is not all of~$X$).

We will do this by considering a test configuration $\X$ for $X$ and proving it is an equivariant contraction of degeneration to the normal cone $\widehat{\X}$ of $\PP(F)\cap X$ inside $X$. We will then compute the (sign of the) Donaldson-Futaki invariant for $\X$ with $A \approx 0$ and relate it to the integral $\mu_{c}(Z, \mathcal{L}_A)$ for $c \approx 1$ thanks to an important estimate of Ross and Thomas. (We further take asymptotics for $g = g(C) \gg 1$ keeping $|\mu(F)|, |\mu(E)|$ bounded, and explain how this corresponds to taking the limit near $C$ ``infinitely special"). The undefined (but standard) terminology above will be briefly recalled at the beginning of section \ref{main_section}.

Let $Q = E/F$. The degeneration $\X$ takes $X \subset \PP(E)$ to a cone over $\PP(F)\cap X $ inside $\PP(F\oplus Q)$ and can be described as follows. Let $\xi \in \Ext^1(Q, F)$ define the extension class of $E$. Scaling the fibres of $F$ with weight $1$ gives a family of extension classes $t\cdot\xi$ for $t \in \C$ and defines a flat family $\mathscr{E} \to \C$ with $\mathscr{E}_t \cong E$ for $t \neq 0$, $\mathscr{E}_0 \cong F\oplus Q$. Projectivising this gives a family $\PP(\E)\to\C$ with a natural action of $\C^*$ which preserves the projection to $\C$. Our variety $X$ can be embedded in the fibre $\PP(\E)_1$ by choosing a form $\varphi \in H^0(S^m \E^*_1)$. We then define $\X$ as the flat closure of the family over $\C\setminus\{0\}$ given by the images of $X \subset \PP(\E)_1$ under $\C^*$. The central fibre $\X_0$ is defined by the form $\varphi|_F := \im(\varphi)$ under the map $S^m E^* \to S^m F^*$, but regarding now $S^m F^*$ as a direct summand of $S^m (F \oplus Q)^*$. Therefore fibrewise over $C$ it is the cone from $\PP(0\oplus Q)$ to $\PP(F)\cap X$ regarded as a closed subscheme of $\PP(\E)_0 \cong \PP(F\oplus Q)$. The same construction goes through when $X$ is a complete intersection; when $X$ is sufficiently general (with respect to $F$) the central fibre $\X_0$ is again a complete intersection.

We give a criterion for the Donaldson-Futaki invariant of $\X$ to be negative. In turn this gives a criterion for the subscheme $\PP(F)\cap X$ to slope-destabilise. The main point here is that computing the Donaldson-Futaki invariant for $\X$ is considerably easier than computing the integral $\mu_{c}(\PP(F)\cap X, L)$, but we can use a result of Ross and Thomas to give a lower bound on the latter.
\begin{defn} Let $F$ be a locally free sheaf on the curve $C$, $r$ a positive integer with $\rk(F) > r$. We define a modified slope function by \[\mu^r(F) = \frac{\deg(F)}{\rk(F)-r}.\]
\end{defn}
\begin{thm}\label{main} Let $E$ be a locally free sheaf on the curve $C$ with $\mathcal{O}_{\PP(E)}(1)$ globally generated, $F \subset E$ a subbundle with $\rk(F) > r$ and
\[\mu^r(F) > \mu^r(E).\]
If $X$ is a general (with respect to $F$) codimension $r$ normal complete intersection of powers of $\mathcal{O}_{\PP(E)}(1)$, then the subscheme $\PP(F) \cap X$ slope-destabilises $(X, \mathcal{L}_A)$ for $A \approx 0$, $c \approx 1$, provided we can take the genus $g$ of $C$ sufficiently large.
\end{thm}
More precisely if $X$ is a complete intersection of elements of $\mathcal{O}_{\PP(E)}(m_i)$ for $i = 1, \dots, r$ we require $g \geq g_0(r,\max_i m_i, \rk(F), \rk(E), |\mu(E)|, |\mu(F)|)$. In particular we should be allowed to let $g \to \infty$ while keeping $\mathcal{O}_{\PP(E)}(1)$ globally generated with a uniform bound on $|\deg(E)|$. Thus the geometric meaning of this condition is that we need to take $C$ arbitrarily special (think of the simple case $E = \mathcal{O}^{\oplus e-1}\oplus L$ with $L$ globally generated and a uniform bound on $\deg L$).\\

It is straightforward to use this result to produce many examples of slope unstable polarised manifolds; simply consider a relative complete intersection inside the projectivisation of a vector bundle whose dual is globally generated over a very special curve and with a subbundle satisfying the above slope instability condition. In particular it becomes almost trivial to produce a wealth of \mbox{K-unstable} general type polarised manifolds (compare with the ``folklore conjecture" disproved in \cite{ross}).

In the last section we concentrate on one example which is perhaps less obvious: we will use Theorem \ref{main} to describe a class of slope unstable blowups of ruled surfaces  (without holomorphic vector fields) for polarisations which differ from those of \cite{rt_diff} Corollary 5.29. Our trick is to regard these surfaces as conic bundles.\\

Finally we observe that the only examples of K\"ahler classes without constant scalar curvature representatives on a general type manifold are obtained via slope instability. We believe it would be quite interesting to find K-unstable, slope stable polarisations on a general type manifold. Note that while proving K-stability directly is essentially impossible, by the work of Ross and Panov \cite{dima} proving slope stability on a surface of general type seems almost within reach, since we should only test against rather special divisors. (We must mention in this connection the fundamental inequalities of Chen \cite{chen} and Song-Weinkove \cite{ben}).\\
\noindent\textbf{Notation.} Throughout the paper we mix the additive and multiplicative notation for line bundles. We denote by $\Tr$ the total weight of a representation of $\C^*$. A typical space of sections over a scheme $Y$ is denoted by $H^0_{Y}$ and its dimension by $h^0_Y$.\\
\noindent\textbf{Acknowledgements.} We thank J. Ross, R. Thomas and G. Sz\'ekelyhidi for their comments on a preliminary version of this paper. The first author is grateful to the Mathematics Department of Pavia University for kind hospitality.

\section{Main result}\label{main_section}
The basic object in algebraic K-stability is a \emph{test configuration}. This is roughly a $\C^*$-equivariant flat family $\Y$ over $\C$ with a fibrewise ample line bundle $\mathscr{L}$ on it, and so a 1-parameter degeneration of the polarised variety $(Y, L)=(\Y_1, \mathscr{L}_1)$ to the polarised scheme $(\Y_0, \mathscr{L}_0)$ (we will not give a precise definition here but see e.g. \cite{rt_alg}).

A test configuration $\Y$ induces an action of $\C^*$ on the central fibre $\Y_0$, and by (equivariant) Riemann-Roch one can expand the dimension and total $\C^*$-weight of the spaces of sections as
\begin{align*}
h^0_{\Y_0}(\mathscr{L}^k_0) = b_0 k^n + b_1 k^{n-1} + o(k^{n-1}),\\
\Tr H^0_{\Y_0}(\mathscr{L}^k_0) = a_0 k^{n+1} + a_1 k^{n} + o(k^{n-1}).
\end{align*}
where $n = \dim Y$. The rational number
\[\mathcal{F}(\Y) = a_0 b_1 - a_1 b_0\]
is called the \emph{Donaldson-Futaki invariant} of $\Y$, and $\Y$ \emph{K-destabilises} if $\mathcal{F}(\Y) < 0$.\\

Let $Z \subset Y$ be a closed subscheme. The \emph{degeneration to the normal cone} is a special test configuration $\Y$ for $Y$ given by the blowup $\operatorname{Bl}_{Z \times\{0\}} Y \times \C$ lifting the natural action of $\C^*$ on the second factor of $Y \times \C$. This is polarised by $L - c E$ where $E$ denotes the exceptional divisor and $c$ is positive, rational and less than the Seshadri constant of $Z$ with respect to the ample line bundle $L$ on $Y$ (which we also identify with its pullback). Ross and Thomas proved
\begin{thm}[\cite{rt_diff} Theorem 4.3]\label{rt_thm} The inequality \eqref{destab} holds if and only if the degeneration to the normal cone K-destabilises (namely its Donaldson-Futaki invariant $\mathcal{F}_c(\Y)$ is negative).
\end{thm}

Let us now go back to the setup described in the Introduction. Recall we are assuming $\mathcal{O}_{\PP(E)}(1)$ is globally generated, hence nef. A standard application of the Nakai-Moishezon criterion and Kleiman's theorem then shows that for any ample $\mathbb{Q}$-divisor $A$ on $C$ the $\mathbb{Q}$-line bundle $\mathcal{L}_A$ is ample. A general complete intersection $X$ of powers of $\mathcal{O}_{\PP(E)}(1)$ is smooth; it is enough for our purposes to take $X$ normal. Its codimension will be denoted by $r$.

Starting with $F \subset E$ a subbundle we form the degeneration $\X$ described in the Introduction. We wish to relate its Donaldson-Futaki invariant to that of the degeneration to the normal cone of $\PP(F)\cap X \subset X$, which we denote by $\widehat{\X}$.
\begin{lem}\label{comparison} Suppose that $\mathcal{F}(\X)$ is negative. Then $\mathcal{F}_c(\widehat{\X})$ is also negative for $c \approx 1$.
\end{lem}
\begin{proof} The test configuration $\X$ (with the line bundle $\mathcal{O}_{\X}(1)$ given by $\X \subset \PP(\mathscr{E})$ as in the Introduction) is an equivariant contraction of the degeneration to the normal cone $\widehat{\X}$ of $\PP(F)\cap X$. To see this let \[p\!: \operatorname{Bl}_{\PP(F)\times\{0\}}\PP(E)\times \C \to \PP(E)\times \C\]
and notice that the semi-ample line bundle $\mathcal{O}_{\widehat{\X}}(1) := \mathcal{O}_{\PP(E)}(1) - p^*[\PP(F)]$ defines a morphism which is an equivariant contraction to the test configuration splitting $\PP(E)$ into $\PP(F\oplus Q)$ with weight $1$ action on $F$ (see \cite{rt_diff} remark 5.14). Then the claim for $\X$ follows since it is the image of the flat closure of $X \subset \PP(E)\times \{1\}$ under the $\C^*$-action, which in turn coincides with degeneration to the normal cone of $\PP(F)\cap X$.

Suppose now that $\mathcal{F}(\X) < 0$. Then since the general fibre $X$ of $\X$ is normal \cite{rt_alg} Proposition 5.1 shows that there exists $a \geq 0$ such that (denoting by $t$ a coordinate on the base $\C$)
\[\Tr(H^0_{\X_0}(\mathcal{O}_{\X}(k))) = \Tr(H^0_{\widehat{\X}}(\mathcal{O}_{\widehat{\X}}(k))/tH^0_{\widehat{\X}}(\mathcal{O}_{\widehat{\X}}(k))) - a k^n + o(k^n).\]
This implies in particular
\[\lim_{c \to 1} \mathcal{F}_c(\widehat{\X}) \leq \mathcal{F}(\X) < 0,\]
and by continuity of the Donaldson-Futaki invariant in the parameter $c$ we obtain $\mathcal{F}_c(\widehat{\X}) < 0$ for $c \approx 1$ as required.
\end{proof}
The rest of this section is devoted to the proof of Theorem \ref{main}. More precisely we compute the derivative $\partial_g \mathcal{F}(\X)$ and show it equals a positive constant times by the difference of slopes $\mu^r(F) - \mu^r(E)$. In view of Lemma \ref{comparison} and Theorem \ref{rt_thm} this will complete the proof of our main result.\\

Note that by the definition of the Donaldson-Futaki invariant we may replace $X$ with the central fibre of $\X_0$ and assume that $E$ splits as $F \oplus Q$ to start with (this new $X$ will not be normal anymore most of the times, but generically for a fixed $F$ it will still be a codimension $r$ complete intersection). Moreover we will assume that $\mathcal{O}_{\PP(E)}(1)$ is ample to start with and set $A = 0$. When $\mathcal{O}_{\PP(E)}(1)$ is not ample our computations below only give the limit $\lim_{A \to 0} \mathcal{F}(\X)$; by continuity of the Donaldson-Futaki invariant in the parameter $A$ this is clearly enough for our purposes.\\

Write $e,f,q$ for the ranks of the bundles $E, F, Q$ respectively. According to the definition we need to expand
\begin{align*}
h^0_X(k) &:= h^0_X(\mathcal{O}_X(k)) = b_0 k^{e-r} + b_1k^{e-r-1} + o(k^{e-r-1}),\\
\Tr_X(k) &:=\Tr H^0_X(\mathcal{O}_X(k)) = a_0 k^{e- r + 1} + a_1 k^{e-r} + o(k^{e-r}).
\end{align*}
To do this we will first find the corresponding expansions for $h^0_{\PP(E)}(\mathcal{O}_{\PP(E)}(k))$ and $\Tr H^0_{\PP(E)}(\mathcal{O}_{\PP(E)}(k))$. We will then use the Koszul resolution of $\mathcal{O}_X$ to write the above polynomials for $X$ as certain linear combinations of shifts of the analogous polynomials for $\PP(E)$. We start by pushing forward to $C$, so by Riemann-Roch we compute
\begin{align}
h^0_{\PP}(k) &:= h^0(\mathcal{O}_{\PP(E)}(k)) = \binom{e -1 + k}{e-1}(-k\mu(E) + 1-g)\\
&\sim -\frac{\mu(E)}{(e-1)!}k^{e} + \left(\frac{1-g}{(e-1)!} - \frac{e\mu(E)}{2(e-2)!}\right)k^{e-1} + o(k^{e-1}).\nonumber
\end{align}
Writing
\[H^0(\mathcal{O}_{\PP(E)}(k)) = \bigoplus^k_{i=0} W_{-i}\]
for the decomposition into semi-invariant spaces of weight $-i$ we have
\[\Tr_{\PP}(k) := \Tr H^0(\mathcal{O}_{\PP(E)}(k)) = -\sum^k_{i = 0}i \dim W_{-i}\]
where
\[W_{-i} = H^0(S^i F^* \otimes S^{k-i} Q^*),\]
and as before Riemann-Roch gives
\[
\dim W_{-i} = \binom{f-1 + i}{f-1}\binom{q-1 + k-i}{q-1}(-i \mu(F) - (k-i)\mu(Q) + 1-g).\]
An elementary way to compute the trace is by absorption, starting from
\begin{align*}
\sum^k_{i=0} i \binom{f-1 + i}{f-1}\binom{q-1 + k-i}{q-1} &= f \sum^{k-1}_{j=0} \binom{f + j}{f}\binom{q-1 + k-1-j}{q-1}\\
&= f \binom{e + k-1}{e} \sim \frac{f}{e!}k^e + \frac{f}{2(e-2)!}k^{e-1} + o(k^{e-1}),
\end{align*}
which gives in turn
\begin{align*}
\sum^k_{i=0} i^2 \binom{f-1 + i}{f-1}&\binom{q-1 + k-i}{q-1} = \\
&= \sum^k_{i=0} i(i-1)\binom{f-1 + i}{f-1}\binom{q-1 + k-i}{q-1} + f \binom{e + k-1}{e}\\
&= f(f+1)\sum^{k-2}_{j=0}\binom{f+1 + j}{f+1}\binom{q-1 + k-2-j}{q-1} + f \binom{e + k-1}{e}\\
&= f(f+1)\binom{e + 1 + k-2}{e+1} + f \binom{e + k-1}{e}\\
&\sim \frac{f(f+1)}{(e+1)!}k^{e+1} + \frac{f(f+1)(e-2)+ 2f}{2 e!}k^e + o(k^{e}).
\end{align*}
Similarly
\begin{align*}
\sum^k_{i=0}i(k-i)\binom{f-1 + i}{f-1}\binom{q-1 + k-i}{q-1} &= f q\binom{e+1 + k - 2}{e+1}\\
&\sim \frac{f q}{(e+1)!}k^{e+1} + \frac{f q (e-2)}{2 e!}k^{e} + o(k^e).
\end{align*}
Thus we find
\begin{align}
\Tr_{\PP}(k) &\sim f\left(\frac{(f+1)\mu(F)}{(e+1)!} + \frac{q\mu(Q)}{(e+1)!}\right)k^{e+1}\\
& + f\left(\frac{(e(f+1)-2f)\mu(F)}{2 e!} + \frac{q(e-2)\mu(Q)}{2 e!} - \frac{(1-g)}{e!}\right)k^e + o(k^e).\nonumber
\end{align}
Suppose now that $X$ is a complete intersection of divisors in $|\mathcal{O}_{\PP(E)}(m_i)|$ for $i = 1, \dots, r$. By the Koszul resolution of $\mathcal{O}_X$ we can compute
\[h^0_X(k) = \sum^r_{s = 0} (-1)^s \sum_{1 \leq i_1 < i_2 < \dots < i_s \leq r} h^0_{\PP}(k - m_{i_1} - \dots - m_{i_s})\]
and similarly, taking into account the shift in the $\Z$-grading given by the \mbox{$\C^*$-action},
\begin{align*}
\Tr_{X}(k) &= \sum^r_{s = 0} (-1)^s \sum_{1 \leq i_1 < i_2 < \dots < i_s \leq r} \Tr_{\PP}(k - m_{i_1} - \dots - m_{i_s})\\& - \sum^r_{s = 0} (-1)^s \sum_{1 \leq i_1 < i_2 < \dots < i_s \leq r} (m_{i_1} + \dots + m_{i_s})h^0_{\PP}(k - m_{i_1} - \dots - m_{i_s}).
\end{align*}
We will need the following combinatorial identities (easily proved by induction) for a polynomial $p(x) = \alpha x^n + o(x^n)$,
\[\sum^r_{s = 0} (-1)^s \sum_{1 \leq i_1 < i_2 < \dots < i_s \leq r} p(k - m_{i_1} - \dots - m_{i_s}) = \binom{n}{r} r! \alpha \prod^r_{i = 1} m_i x^{n-r} + o(x^{n-r}),\]
\begin{align*}\sum^r_{s = 0} (-1)^s \sum_{1 \leq i_1 < i_2 < \dots < i_s \leq r} (m_{i_1} + \dots + m_{i_s}) &p(k - m_{i_1} - \dots - m_{i_s})=\\&= -\binom{n}{r-1} r! \alpha \prod^r_{i = 1} m_i \,x^{n-r+1} + o(x^{n-r+1}).
\end{align*}
Computing with these identities we find
\begin{align*}
b_0 &= -\binom{e}{r} r! \prod^r_{i = 1} m_i \frac{\mu(E)}{(e-1)!}=-\prod^r_{i = 1} m_i\frac{e\mu(E)}{(e-r)!},\\
a_0 &= \binom{e+1}{r} r! \prod^r_{i = 1} m_i \,f\left(\frac{(f+1)\mu(F)}{(e+1)!} + \frac{q\mu(Q)}{(e+1)!}\right) - \binom{e}{r-1} r! \prod^r_{i = 1} m_i \,\frac{\mu(E)}{(e-1)!}\\
&= \prod^r_{i = 1} m_i \frac{f\mu(f) + e(f-r)\mu(E)}{(e+1-r)!}.
\end{align*}
We now take derivatives with respect to the genus $g$. We apply the same combinatorial identities, the key fact being that the higher order coefficients $a_0, b_0$ do not depend on $g$ and therefore give no further contributions to the lower order terms,
\begin{align*}
\partial_g b_0 &= \partial_g a_0 = 0,\\
\partial_g b_1 &= -\binom{e-1}{r} r! \prod^r_{i = 1} m_i \frac{1}{(e-1)!} = -\prod^r_{i = 1} m_i \frac{1}{(e-1-r)!},\\
\partial_g a_1 &= \binom{e}{r} r! \prod^r_{i = 1} m_i \frac{1}{e!} - \binom{e-1}{r-1} r! \prod^r_{i = 1} m_i \frac{1}{(e-1)!} = \prod^r_{i = 1} m_i \frac{f-r}{(e-r)!}.
\end{align*}
By definition $\mathcal{F}(\X) = a_0 b_1 - a_1 b_0$, so
\begin{align*}
\partial_g \mathcal{F}(\X) &= a_0\partial_g b_1 - \partial_g a_1 b_0\\
&= \left(\prod^r_{i = 1} m_i\right)^2 \frac{(f-r)e \mu(E)-(e-r)f\mu(f)}{(e-r)!(e-r+1)!}\\
&= \left(\prod^r_{i = 1} m_i\right)^2 \frac{\mu^r(E)-\mu^r(F)}{(e-r)!(e-r+1)!}.
\end{align*}
This completes the proof of Theorem \ref{main}.
\section{An example}
In this section we give new examples of slope unstable blowups of ruled surfaces. Our polarisations differ from those of \cite{rt_diff} Corollary 5.29 as the exceptional divisors all have the same large area. Our trick is to regard these blowups as conic bundles.

Let $C$ be a genus $g$ \emph{hyperelliptic} Riemann surface with hyperelliptic divisor $H$ and $D \neq H$ a divisor on $C$ such that $\mathcal{O}_{C}(D)$ is non trivial and globally generated (the same argument would also apply to any Riemann surface with a nontrivial globally generated divisor $H$ of degree less than $\frac{g+2}{3}$, but we only write it down in the hyperelliptic case for simplicity). In particular (since the hyperelliptic divisor is unique) we have $\deg(D) > \deg(H) = 2$. Consider the vector bundle \begin{equation}
E= \mathcal{O}_{C} \oplus \mathcal{O}_{C}(-H) \oplus \mathcal{O}_{C}(-D),
\end{equation}
and the projective bundle $\PP(E)$ over $C$. The linear system $|\mathcal{O}_{\PP(E)}(2)|$ contains a smooth irreducible surface $S$ and the induced map $\pi\!: S \rightarrow C$ is a fibration by degree 2 plane curves, i.e. a conic bundle.
\begin{lem} A generic surface $S \in |\mathcal{O}_{\PP(E)}(2)|$ is the blowup of some ruled surface $\overline{S} \to C$ in $2 \deg(D) + 4$ points.
\end{lem}
\begin{proof} Let us first show that $\chi(\mathcal{O}_S)=1-g$ and $K^2_S=8(1-g)-2\deg(D)-4$. Consider the exact sequence
\begin{equation}\label{def S}
0 \rightarrow \mathcal{O}_{\PP(E)}(-2) \rightarrow \mathcal{O}_{\PP(E)} \rightarrow \mathcal{O}_S \rightarrow 0
\end{equation}
It follows from \cite{hart}, Exercises III.8.1 and III.8.4 that $\chi(\mathcal{O}_{\PP(E)})=\chi(\mathcal{O}_{C})$ and $\chi(\mathcal{O}_{\PP(E)}(-2))=0$, thus $\chi(\mathcal{O}_S)=\chi(\mathcal{O}_{C})=1-g$.

We also get $K_{\PP(E)}=\mathcal{O}_{\PP(E)}(-3)\otimes \pi^*(\Lambda^3 E^* + K_{C})$ and, by adjunction,
\[K_S=\mathcal{O}_{S}(-1)\otimes \pi^*(\Lambda^3 E^* + K_{C})=\mathcal{O}_{S}(-1)\otimes \pi^*(\mathcal{O}_{C}(H+D) + K_{C})\]
By \cite{bea}, Theorem I.4 we find
\[K^2_S=\chi(\mathcal{O}_S(-2K_S))-2\chi(\mathcal{O}_S(-K_S))+\chi(\mathcal{O}_S)\]
Similarly to $\chi(\mathcal{O}_{\PP(E)})$, we can then compute \[\chi(\mathcal{O}_S(-2K_S))=\chi(S^2E^*(-2D-2H-2K_{C}))-\chi(\mathcal{O}_{C}(-2D-2H-2K_{\C}))\]
and
\[\chi(\mathcal{O}_S(-K_S))=\chi(E^*(-D-K_{C})).\]
By Riemann-Roch on $C$ we conclude that
\[K_S^2=8(1-g)-2\deg(D + H) = 8(1-g) - 2\deg(D) - 4.\]

Let us now show that $\pi\!: S \to C$ has exactly $2 \deg(D) + 4$ nodal fibres. The topological Euler characteristic is given by
\[e(S)=2e(C)+\sum_{s \in C} (e(F_s)-2)\]
where $F_s$ is the fibre over the point $s \in C$ (see \cite{bea}, Proposition X.10). Since $S$ is smooth a local computation shows that there are no double fibres, and the only singular fibres that can occur are the union of 2 irreducible rational components meeting in one point. Such fibres have topological Euler number equal to 3, thus the number $e(S)-2e(C)$ is precisely the number of singular fibres. We can compute this by Noether's Formula:
\[e(S)=12 \chi(\mathcal{O}_S)-K_S^2=4(1-g)+2\deg(D)+4\]
thus we have precisely $2 \deg(D) + 4$ singular fibres.

The proof will be complete if we can show that any component of a singular fibre is a $-1$ curve, since then we can contract one component for any singular fibre to obtain a ruled surface $\overline{S}\to C$. Let $F = D_1 + D_2$ be a singular fibre (so $D_i$, $i = 1, 2$ are its rational irreducible components). By \cite{bea} Proposition VIII.3 we have $D^2_i < 0$, and $F^2 = 0$ implies $D^2_1 + D^2_2 = -2 D_1 . D_2 = - 2$, so $D^2_i = -1$ for $i = 1, 2$.
\end{proof}
\begin{prop}
A generic $S \in |\mathcal{O}_{\PP(E)}(2)|$ has no holomorphic vector fields.
\end{prop}
\begin{proof} Let us first show that any holomorphic vector field on $S$ extends to one on $\PP(E)$. Regarding $H^0(S, T_S)$ naturally as a subspace of $H^0(S, (T_{\PP(E)})_{|_S})$ it is enough to prove that the restriction map $H^0(\PP(E), T_{\PP(E)}) \to H^0(S, (T_{\PP(E)})_{|_S})$ is onto. Its cokernel is contained in $H^1(\PP(E), T_{\PP(E)}(-S))$. We will prove that the latter group vanishes. By the exact sequence defining the relative tangent bundle
\[0 \to T_{\PP(E)|C} \to T_{\PP(E)}\to \pi^* T_{C} \to 0\]
it is enough to prove the vanishing of the groups $H^1(\PP(E), T_{\PP(E)|C}(-S))$ and $H^1(\PP(E), \pi^* T_{C}(-S))$. This can be verified again using \cite{hart} Exercises III.8.1,~III.8.4.

It remains to be checked that there are no holomorphic vector fields on $\PP(E)$ preserving $S$. The infinitesimal action of vector fields on $\PP(E)$ on $S$ is given by a map $\varphi\!:\operatorname{End}(E)/\C I \to H^0(S^2 E^*)$ (writing $\operatorname{End}/\C I$ for dividing out by the addition of multiples of the identity). To see this we can represent $S$ by $Q \in H^0(S^2 E^*)$ uniquely up to rescaling; then $\varphi$ is given on $A \in \operatorname{End}(E)/\C I$ by $\varphi(A) = A^*Q + Q A$ (regarding $A, Q$ as matrices in a natural way). We write
\begin{equation}
A = \left(
\begin{matrix}
 A_0& v_A \\
 0  & 0
\end{matrix}
\right)
\end{equation}
where $A_0 \in \operatorname{End}(\mathcal{O}_{C}\oplus\mathcal{O}_{C}(-H))$, $^tv_A \in H^0(\mathcal{O}_{C}(D)\oplus\mathcal{O}_{C}(D-H))$. The first $0$ in the lower row means that $\mathcal{O}_{C}(-D)$ and $\mathcal{O}_{C}(H-D)$ have no global sections; the second is only the choice of a representative modulo $\C I$. Similarly we can write
\begin{equation}
Q = \left(
\begin{matrix}
Q_0     & v_Q \\
 ^tv_Q  & q
\end{matrix}
\right)
\end{equation}
where $Q_0 \in H^0(S^2(\mathcal{O}_{C}\oplus\mathcal{O}_{C}(H)))$, $^tv_Q \in H^0(\mathcal{O}_{C}(D)\oplus\mathcal{O}_{C}(D+H))$, $q \in H^0(\mathcal{O}_{C}(2D))$. Therefore
\begin{equation}
\varphi(A) = \left(
\begin{matrix}
A_0^*Q_0   + Q_0 A_0   & A^*_0 v_Q + Q_0 v_A \\
^tv_A Q_0 + {^tv_Q} A_0 & 2\,{^tv_A}v_Q
\end{matrix}
\right).
\end{equation}
The condition $\varphi(A) = 0$ implies ${^tv_A}v_Q = 0$. A slight twist of the ``free pencil trick" (see e.g. \cite{ACGH} p. 126) shows that for generic $v_Q$ we must have $v_A = 0$. In turn this implies $A^*_0 v_Q = 0$. The proof will be complete if we can show that for generic $v_Q$ this forces $A_0 = 0$ since then $\varphi(A) = 0$ implies $A = 0$. We compute
\begin{equation}
A^*_0 v_Q = \left(
\begin{matrix}
\lambda & 0 \\
   t    & \mu
\end{matrix}
\right)\left(
\begin{matrix}
s_1 \\ s_2
\end{matrix}
\right) = \left(
\begin{matrix}
\lambda s_1 \\
 t s_1 + \mu s_2
\end{matrix}
\right)
\end{equation}
where $t \in H^0(\mathcal{O}_{C}(H))$, $s_1 \in H^0(\mathcal{O}_{C}(D))$ and $s_2\in H^0(\mathcal{O}_{C}(D+H))$ (the upper right $0$ in $A^*_0$ expressing that $\mathcal{O}_{C}(-H)$ has no global sections). We would then have $\lambda = 0$ and $\mu s_2$ would lie in the image of the multiplication map $H^0(\mathcal{O}_{C}(H)) \Rt{\otimes s_1} H^0(\mathcal{O}_{C}(D + H))$. For generic $s_1, s_2$ this is impossible for dimension reasons.
\end{proof}
Finally it is clear that $F = \mathcal{O}_C\oplus\mathcal{O}_C(-H) \subset E$ satisfies $\mu^1(F) > \mu^1(E)$. According to Theorem \ref{main} then the divisor $\PP(F)\cap S$ slope-destabilises $S$ for $c \approx 1$, $A \approx 0$ and all large $g$. Note that in this particular case a more careful computation shows that taking $g > 16$ is enough to give instability for all $D$.

\vspace{.5cm}
Max Planck Institute for Mathematics\\
{\tt stoppa@mpim-bonn.mpg.de}
\vspace{.5cm}\\
Dipartimento di Matematica ``F. Casorati"\\
Universit\`a degli Studi di Pavia\\
{\tt elisa.tenni@unipv.it}

\begin{thebibliography}{}
\bibitem{ACGH} E. Arbarello, M. Cornalba, P. A. Griffiths, J. Harris, \emph{Geometry of Algebraic Curves, vol.1}, Springer, New York - Berlin - Heidelberg - Tokyo, 1985, 386 pp.
\bibitem{bea} A. Beauville, \emph{Complex Algebraic Surfaces}, Second Edition, Cambridge University press, Cambdridge, 1996, 132 pp.
\bibitem{chen} X. X. Chen, \emph{On the lower bound of the Mabuchi energy and its application.} Internat. Math. Res. Notices \textbf{12} (2000), 607--623.
\bibitem{chen_tian} X. X. Chen and G. Tian, \emph{Geometry of K\"ahler metrics and foliations by holomorphic discs}, Publ. Math. Inst. Hautes Études Sci. \textbf{107}  (2008), 1--107.
\bibitem{don_calabi} S. K. Donaldson, \emph{Lower bounds on the Calabi functional}, J. Differential Geom. \textbf{70} (2005), no. 3, 453--472.
\bibitem{hart} R. Hartshorne, \emph{Algebraic geometry.} Graduate Texts in Mathematics, No. 52. Springer-Verlag, New York-Heidelberg, 1977, xvi+496 pp.
\bibitem{dima} D. Panov and J. Ross, \emph{Slope stability and exceptional divisors of high genus}, Math. Ann. \textbf{343} no. 1, 2009.
\bibitem{ross} J. Ross, \emph{Unstable products of smooth curves}, Invent Math. Vol. \textbf{165} (2006), no. 1, 153-162.
\bibitem{rt_diff} J. Ross and R. P. Thomas, \emph{An obstruction to the existence of constant scalar curvature K\"ahler metrics}, Jour. Diff. Geom. \textbf{72} (2006), 429--466.
\bibitem{rt_alg} J. Ross and R. P. Thomas, \emph{A study of the Hilbert-Mumford criterion for the stability of projective varieties}, J. Algebraic Geom.  \textbf{16} (2007),  no. 2, 201--255.
\bibitem{ben} J. Song, B. Weinkove, \emph{On the convergence and singularities of the $J$-flow with applications to the Mabuchi energy}, Comm. Pure Appl. Math. \textbf{61} (2008),  no. 2, 210--229.
\bibitem{twisted} J. Stoppa, \emph{Twisted constant scalar curvature K\"ahler metrics and K\"ahler slope stability}, preprint arXiv:0804.0414v1 [math.DG]
\end{thebibliography}
\end{document}